\documentclass[11pt,a4paper,twoside]{article}
\usepackage{amsmath}
\usepackage{amsfonts}
\usepackage{algorithm}
\usepackage{algorithmic}
\usepackage{booktabs}
\usepackage[dvipsnames]{xcolor}
\usepackage{enumerate}
\usepackage{lscape}
\usepackage{longtable}
\usepackage{url}
\usepackage{rotating}
\usepackage{titlesec}
\usepackage[titletoc]{appendix}

\usepackage[labelfont=sl,textfont=sl]{subcaption} 
\usepackage[font={small,sl},labelsep=period]{caption} 

\newtheorem{theorem}{Theorem}[section]

\newtheorem{corollary}{Corollary}[section]

\newtheorem{lemma}{Lemma}[section]

\newtheorem{proposition}{Proposition}[section]

\newtheorem{assumption}{Assumption}[section]

\newenvironment{proof}[1][Proof]{\textbf{#1.} }{\hfill$\Box$}

\newcommand{\R}{\mathbb{R}}

\newcommand{\T}{\mathrm{T}}
\newcommand{\cS}{{\cal S}}

\newcommand{\cU}{{\cal U}}

\newcommand{\epsg}{\epsilon_g}
\newcommand{\epsR}{\epsilon_R}

\newcommand{\cO}{{\cal O}}

\DeclareMathOperator*{\minimize}{\mathrm{minimize}}

\def \st {\operatorname*{subject\ to\ }}

\newcommand{\rev}[1]{\textcolor{black}{#1}}
\newcommand{\revrev}[1]{\textcolor{black}{#1}}

\title{Complexity analysis of regularization methods 
for implicitly constrained least squares}

\author{Akwum Onwunta\thanks{Department of Industrial and Systems Engineering,
Lehigh University,
200 West Packer Avenue, Bethlehem, PA 18015-1582, USA
(\texttt{ako221@lehigh.edu}).} 
\and 
Cl\'ement W. Royer\thanks{LAMSADE, CNRS, Universit\'e Paris Dauphine-PSL, 
Place du Mar\'echal de Lattre de Tassigny, 75016 Paris, France 
(\texttt{clement.royer@lamsade.dauphine.fr}).}
}

\begin{document}

\maketitle

\begin{abstract}
	Optimization problems constrained by partial differential equations 
	(PDEs) naturally arise in scientific computing, as those constraints 
	often model physical systems or the simulation thereof. In an implicitly 
	constrained approach, the constraints are incorporated into the objective 
	through a reduced formulation. To this end, a numerical procedure is 
	typically applied to solve the constraint system, and efficient numerical 
	routines with quantifiable cost have long been developed for that 
	purpose. Meanwhile, the field of complexity in optimization, that estimates 
	the cost of an optimization algorithm, has received significant attention in 
	the literature, with most of the focus being on unconstrained or 
	explicitly constrained problems.
		In this paper, we analyze an algorithmic framework based on quadratic 
	regularization for implicitly constrained nonlinear least squares. By 
	leveraging adjoint formulations, we can quantify the worst-case cost 
	of our method to reach an approximate stationary point of the optimization 
	problem. Our definition of such points exploits the least-squares structure 
	of the objective, and provides new complexity insights even in the 
	unconstrained setting. Numerical experiments conducted on PDE-constrained 
	optimization problems demonstrate the efficiency of the proposed framework.
\end{abstract}

\section{Introduction}
\label{sec:intro}


PDE-constrained optimization problems arise in various scientific and 
engineering fields \rev{when searching for the optimal distribution of 
a given quantity that satisfies} physical or mathematical laws described by 
PDEs, such as heat conduction or electromagnetic 
waves~\cite{HAntil_DPKouri_MDLacasse_DRidzal_2016EDS,Fredi_2010,
MHinze_RPinnau_MUlbrich_SUlbrich_2009, MUlbrich_2011}. Similar constrained 
formulations have also received recent interest from the machine learning 
community, as they opened new possibilities for building neural network 
architectures~\cite{LRuthotto_EHaber_2020}. A popular approach to handle 
PDE constraints is the so-called reduced formulation, in which the constraints 
are incorporated into the objective and become \emph{implicit}. By properly 
accounting for the presence of these constraints while computing derivatives, 
it becomes possible to generalize unconstrained optimization techniques to 
the implicitly constrained setting~\cite{Fredi_2010}. 
\revrev{This paradigm resembles that of manifold 
optimization~\cite{NBoumal_2023}, a field of study that has gained significant 
attention over the past decade and that produced the so-called \emph{implicit} 
methods by accounting for constraints through Riemannian geometry.}

\revrev{
Popular algorithms for both PDE-constrained and manifold optimization are 
based on variants of the trust-region 
method~\cite{MHeinkenschloss_DRidzal_2014,CGBaker_PAAbsil_KAGallivan_2008}. 
Although these algorithms can be equipped with theoretical guarantees and 
efficiently implemented, they are typically designed with general, 
nonlinear objective functions and constraints in mind. Moreover, the cases 
in which these methods are used often involve least-squares objectives, e.g. 
when those problems amount to fitting a model to some observations or data. 
Thus, one naturally wonders whether the least-squares structure can be 
exploited in order to design algorithms tailored to these problems. In 
the context of PDE-constrained optimization problems with explicit constraints, 
methods tailored to least-squares formulations such as Gauss-Newton  
techniques~\cite{JNocedal_SJWright_2006}, have shown promising performance 
when compared to generic trust-region 
approaches~\cite{EBergou_YDiouane_VKungurtsev_CWRoyer_2021}. Still, it is 
unclear whether exploiting this structure can be beneficial in an 
implicitly-constrained setting.
}

\revrev{Meanwhile, \emph{worst-case complexity} has emerged as a way to 
analyze algorithmic performance over the past decade, especially in the 
nonconvex optimization community~\cite{CCartis_NIMGould_PhLToint_2022}.} A 
complexity bound characterizes the worst-case performance of a given 
optimization scheme according to a performance metric (e.g., number of 
iterations, derivative evaluations, etc) and a stopping criterion (e.g., 
approximate optimality, predefined budget, etc). Recent progress in the area 
has switched from designing optimization techniques with complexity guarantees 
in mind to studying popular algorithmic frameworks through the prism of 
complexity, with several results focusing on the least-squares 
setting~\cite{EBergou_YDiouane_VKungurtsev_2020,
EBergou_YDiouane_VKungurtsev_CWRoyer_2022,CCartis_NIMGould_PhLToint_2013d,
NIMGould_TRees_JAScott_2017}, \revrev{as well as manifold optimization 
algorithms~\cite{AAgarwal_NBoumal_BBullins_CCartis_2021,NBoumal_2023}. Despite 
these recent advances, complexity results remain unexplored in PDE-constrained 
optimization, particularly in implicitly-constrained optimization problems.}

In this paper, we study an algorithmic framework for least-squares problems 
with implicit constraints. Our approach leverages the particular structure 
of the objective in order to compute derivatives \revrev{through a carefully 
designed adjoint equation}.  Under standard assumptions for this 
class of methods, we establish complexity guarantees for our framework. In 
a departure from standard literature, our analysis is based on a recently 
proposed stationarity criterion for least-squares 
problems~\cite{CCartis_NIMGould_PhLToint_2013d}. To the best of our knowledge, 
these results are the first of their kind for implicitly constrained problems. 
In addition, our complexity results improve over bounds recently obtained 
in the unconstrained setting~\cite{EBergou_YDiouane_VKungurtsev_CWRoyer_2022}, 
thereby advancing our understanding of complexity guarantees for 
least-squares problems. Numerical experiments on PDE-constrained problems 
illustrate the practical relevance of the proposed stationarity criterion, 
and show that our framework handles both small and large residual problems, 
as well as nonlinearity in the implicit constraints.

The rest of this paper is organized as follows. In Section~\ref{sec:setup}, 
we present our formulation of interest, and discuss how its least-squares 
structure is used to design our algorithmic framework. We establish complexity 
guarantees for several instances of our proposed method in 
Section~\ref{sec:wcc}. In Section~\ref{sec:num}, we investigate the performance 
of our algorithm on classical benchmark problems from PDE-constrained 
optimization. We finally summarize our work in Section~\ref{sec:conc}.

\section{Least-squares optimization with implicit constraints}
\label{sec:setup}

In this paper, we discuss algorithms for least-squares problems of the form
\begin{equation}
\label{eq:genpb}
	\min_{u \in \R^{\rev{n}}} J(y,u):=\frac{1}{2}\|R(y,u)\|^2 
	\quad \st \quad c(y,u) = 0,
\end{equation}
\rev{where $R: \R^{n_y} \times \R^n \rightarrow \R^m$ is a vector-valued 
function, $\|\cdot\|$ denotes the Euclidean vector norm}
\footnote{\rev{Throughout this paper, we use $\|\cdot\|$ for both the vector 
Euclidean norm and the induced spectral norm on matrices.}}
\rev{
and $c: \R^{n_y} \times \R^n \rightarrow \R^p$.
}
\rev{Problem~\eqref{eq:genpb} involves both the variable $u \in \R^n$ (typically 
representing a control on a 
given system) as well as a vector of auxiliary variables $y \in \R^{n_y}$} 
(often reflecting the state of the system). We are interested in problems where it is possible to 
(numerically) solve the constraint equation $c(y,u)=0$ to obtain a unique 
solution $y$ given $u$. Problem~\eqref{eq:genpb} can then be reformulated as
\begin{equation}
\label{eq:reducedpb}
	\min_{u \in \R^n} J(y(u),u) = \frac{1}{2}\|R(y(u),u)\|^2,
\end{equation}
where the constraint arises implicitly in the 
formulation~\cite{MHeinkenschloss_2018}. In PDE-constrained 
optimization, the constraint is a PDE, that can be solved given 
a value for the control vector $u$ to yield a state vector $y(u)$. In 
that setting, problem~\eqref{eq:reducedpb} is often called the 
\emph{reduced formulation}~\cite[Chapter 1]{Fredi_2010}. We are particularly 
interested in leveraging the least-squares nature of problem~\eqref{eq:genpb}.
To this end, we describe in Section~\ref{subsec:adjoints} how derivatives can 
be computed by the adjoint approach for problem~\eqref{eq:reducedpb} while 
leveraging the problem structure. Our algorithm is then given in 
Section~\ref{subsec:algo}.

\subsection{\rev{Adjoint formula} for a least-squares problem}
\label{subsec:adjoints}

In this section, we derive an adjoint formula associated with the reduced 
formulation~\eqref{eq:reducedpb}. Even though the analysis relies on standard 
arguments, to the best of our knowledge the formulas for the least-squares 
setting are rather unusual in the literature. We believe that they may be of 
independent interest, and therefore we provide the full derivation below.

To this end, we make the following assumption on our problem, which is a 
simplified version of a standard requirement in 
implicitly constrained problems~\cite{MHeinkenschloss_2018}. 
\rev{Similar assumptions are found in the PDE-constrained optimization 
literature, see e.g.,~\cite{HAntil_DPKouri_MDLacasse_DRidzal_2016EDS,
MHinze_RPinnau_MUlbrich_SUlbrich_2009}.}

\begin{assumption}
\label{as:reduced}
	For any $u \in \R^{\rev{n}}$, the following properties hold.
	\begin{enumerate}[(i)]
		\item There exists a unique vector $y(u) \rev{\in \R^{n_y}}$ such that $c(y,u)=0$. 
		\item The functions $J$ and $c$ are twice continuously 
		differentiable.
		\item The Jacobian of $c$ with respect to its first argument, 
		denoted by $c_y(\cdot,\cdot)$, is invertible at any $(y,u)$ such that 
		$c(y,u)=0$.
	\end{enumerate}
\end{assumption}

\rev{We now introduce derivative notations for the rest of the paper. We let 
$\hat{J}(u):=J(y(u),u)$ and $\hat{R}(u) := R(y(u),u)$ denote the reduced 
objective function of~\eqref{eq:reducedpb} and its associated residual 
function, respectively. For any $u \in \R^n$, we let 
$\nabla \hat{J}(u) \in \R^n$ denote the gradient of $\hat{J}$ with respect 
to $u$, while for any pair $(y,u) \in \R^{n_y} \times \R^n$, we let 
$\nabla_y J(y,u) \in \R^{n_y}$ and $\nabla_u J(y,u) \in \R^n$ denote 
the partial gradients of $J$ at $(y,u)$ with respect to its first and 
second arguments, respectively.}

\rev{It then follows} from the chain rule that
\begin{equation}
\label{eq:gradhatJu}
	\nabla \hat{J}(u) = \nabla_u J(y(u),u) + c_u(y(u),u)^\T \lambda(u),
\end{equation}
where $\lambda(u)$ is a solution of the so-called adjoint equation
\[
	c_y(y(u),u)^\T \lambda = -\nabla_y J(y(u),u).
\]
This observation is at the heart of the adjoint method. In our case, 
we can leverage the least-squares \rev{nature} of our problem to 
decompose~\eqref{eq:gradhatJu} further. Indeed, the derivatives of 
$J$ with respect to its first and second arguments are given by
\begin{equation}
\label{eq:leastsquaresder}
	\left\{
		\begin{array}{lll}
			\nabla_y J(y(u),u) &= &G_y(y(u),u)^\T R(y(u),u) \\
			\nabla_u J(y(u),u) &= &G_u(y(u),u)^\T R(y(u),u), \\
		\end{array}
	\right.
\end{equation}
where $G_y(y,u)$ and $G_u(y,u)$ are the Jacobian matrices of \rev{$R$} with respect to $y$ 
and $u$\rev{, respectively}. Plugging these expressions in~\eqref{eq:gradhatJu}, we obtain
\begin{eqnarray*}
	\nabla \hat{J}(u) &= &\nabla_u J(y(u),u) 
+ c_u(y(u),u)^\T \lambda(u) \\
	&= &G_u(y(u),u)^\T R(y(u),u) + c_u(y(u),u)^\T \lambda(u) \\
	&= &G_u(y(u),u)^\T R(y(u),u) 
	- c_u(y(u),u)^\T \left[c_y(y(u),u)^\T\right]^{\dagger} \nabla_y J(y(u),u) \\
	&= &G_u(y(u),u)^\T R(y(u),u) 
	- c_u(y(u),u)^\T \left[c_y(y(u),u)^\T\right]^{\dagger}  G_y(y(u),u)^\T R(y(u),u) \\
	&= &\left[ G_u(y(u),u) - G_y(y(u),u)  c_y(y(u),u)^{\dagger} c_u(y(u),u) \right]^\T R(y(u),u) \\
	&= &\left[ G_u(y(u),u) - G_y(y(u),u)  c_y(y(u),u)^{\dagger} c_u(y(u),u) \right]^\T R(y(u),u),
\end{eqnarray*}
where $A^{\dagger}$ denotes the Moore-Penrose pseudo-inverse\rev{.}
According to this expression, we can identify the Jacobian of \rev{$\hat{R}$ at $u$}, 
which we denote by $\hat{G}(u)$, as
\begin{equation}
\label{eq:hatG}
	\hat{G}(u) := \rev{G_u(y(u),u) - G_y(y(u),u)}  c_y(y(u),u)^{\dagger} c_u(y(u),u).
\end{equation}

Algorithm~\ref{alg:jacLS} summarizes the analysis below, and describes the 
adjoint method to compute derivative information when the objective has a 
least-squares structure.

\begin{algorithm}[ht!]
\caption{Adjoint equations for implicitly constrained least squares}
\label{alg:jacLS}
\begin{algorithmic}[1]
\REQUIRE Point of interest \rev{$u \in \R^n$ and vector $y(u) \in \R^{n_y}$} 
such that $c(y(u),u)=0$.
\STATE Solve the equation
\begin{equation}
\label{eq:zetaeq}
	c_y(y(u),u) \zeta(u) = - \rev{c_u(y(u),u)}
\end{equation}
for $\zeta(u)$.
\STATE Compute $\hat{G}(u) = \rev{G_u(y(u),u) + G_y(y(u),u)} \zeta(u)$.
\end{algorithmic}
\end{algorithm}

In our main algorithm as well as its implementation, we will rely on 
Algorithm~\ref{alg:jacLS} to compute derivative information.

\subsection{Algorithmic framework}
\label{subsec:algo}

Our optimization procedure is described in Algorithm~\ref{alg:gn}. This method 
builds on the Levenberg-Marquardt paradigm~\cite{JNocedal_SJWright_2006} and 
more generally on quadratic regularization techniques.

At every iteration, the method computes a tentative step by approximately 
minimizing a quadratic model of the function. This step is then accepted or 
rejected depending on whether it produces sufficient function reduction 
compared to that predicted by the model. The $k$th iteration of 
Algorithm~\ref{alg:gn} will be called \emph{successful} if $u_{k+1}\neq u_k$, 
and \emph{unsuccessful} otherwise. The model is defined using a 
quadratic regularization parameter, that is decreased on successful 
iterations and increased on unsuccessful ones.

\begin{algorithm}[ht!]
\caption{Regularization method for constrained least squares}
\label{alg:gn}
\begin{algorithmic}[1]
\REQUIRE Initial iterate $u_0 \in \R^n$; initial parameter $\gamma_0 >0$; 
minimum regularization parameter $0 <\gamma_{\min} \le \gamma_0$; step 
acceptance threshold $\eta \in (0,1)$.
\smallskip
\hrule
\smallskip
\STATE Solve the constraint $c(y,u_0)=0$ for $y$ to obtain $y(u_0)$.
\STATE \rev{Evaluate $R_0=\hat{R}(u_0)$ and call Algorithm~\ref{alg:jacLS} 
to compute $G_0=\hat{G}(u_0)$.} 
\FOR{$k=0,1,2,\dotsc$} 
\STATE Compute a step $s_k$ as an approximate solution to the following 
problem
\begin{equation}
\label{eq:lssubpb}
	\min_{s \in \R^n} m_k(u_k+s):= \frac{1}{2}\|R_k\|^2 + g_k^\T s 
	+ \frac{1}{2}s^\T (H_k + \gamma_k I) s,
\end{equation}
where $g_k=G_k^\T R_k$ is the gradient of $\hat{J}$ at $u_k$ and 
$H_k \in \R^{n \times n}$ is a symmetric matrix.
\STATE Solve the constraint $c(y,u_k+s_k)=0$ for $y$ to obtain $y(u_k+s_k)$.
\STATE Compute the ratio of actual to predicted decrease in $f$ defined as
\[
\rho_k \gets \frac{\hat{J}(u_k)-\hat{J}(u_k+s_k)}{m_k(u_k)-m_k(u_k+s_k)}.
\]
\IF{$\rho_k \ge \eta$}
\STATE Set $u_{k+1} \gets u_k+s_k$ and $\gamma_{k+1} \gets \max\{0.5\gamma_k,\gamma_{\min}\}$.
\STATE Solve the constraint $c(y,u_{k+1})=0$ for $y$ to obtain $y(u_{k+1})$.
\STATE \rev{Evaluate $R_{k+1}=\hat{R}(u_{k+1})$ and call Algorithm~\ref{alg:gn} 
to compute $G_{k+1}=\hat{G}(u_{k+1})$.}
\ELSE \STATE Set $u_{k+1} \gets u_k$ and $\gamma_{k+1} \gets 2 \gamma_k$.
\ENDIF
\ENDFOR
\end{algorithmic}
\end{algorithm}

Note that our method can be instantiated in several ways, depending on the 
way $H_k$ is computed at every iteration, and on how the 
subproblem~\eqref{eq:lssubpb} is solved. In this paper, we assume that $H_k$ 
is built using only first-order derivative information, and consider two 
specific cases. When $H_k$ is the zero matrix, then the method can be viewed as 
a gradient method. When $H_k = G_k^\T G_k$, the method is a regularized 
Gauss-Newton iteration, similar to the Levenberg-Marquardt method. Although 
we focus on the two aforementioned cases in this paper, we point out that 
other formulae such as quasi-Newton updates~\cite{JNocedal_SJWright_2006} 
could also be used without the need for second-order information. As for the 
subproblem solve, we provide guarantees for exact and inexact variants of our 
framework in the next section.

\section{Complexity analysis}
\label{sec:wcc}

In this section, we investigate the theoretical properties of 
Algorithm~\ref{alg:gn} through the lens of complexity. More precisely, we are 
interested in bounding the effort needed to reach a vector $u_k$ such that
\begin{equation}
\label{eq:epspt}
	\|\rev{R(u_k)}\| \le \epsR
	\quad \mbox{or} \quad 
	\frac{\|\rev{\hat{G}(u_k)^\T R(u_k)}\|}{\|\rev{R(u_k)}\|} \le \epsg.
\end{equation}
Condition~\eqref{eq:epspt} implicitly distinguishes between two kinds of 
approximate stationary points. When possible, one would ideally compute a 
point for which the residuals are small, or a stationary point for the norm of 
the residual. This \emph{scaled gradient} condition was previously used for 
establishing complexity guarantees for algorithms applied to nonlinear 
least-squares problems~\cite{EBergou_YDiouane_VKungurtsev_CWRoyer_2022,
CCartis_NIMGould_PhLToint_2013d,NIMGould_TRees_JAScott_2019}.

Section~\ref{subsec:iter} provides an iteration 
complexity bound for all instances of the algorithm based on the 
condition~\eqref{eq:epspt}, assuming that all linear algebra operations are 
performed in an exact fashion. We then consider inexact variants of our 
algorithms, and derive the associated complexity results in 
Section~\ref{subsec:inexact}.

\subsection{Iteration complexity}
\label{subsec:iter}

We begin by a series of assumptions regarding the reduced 
formulation~\eqref{eq:reducedpb}. \rev{These assumptions are classical in 
complexity analysis on nonlinear optimization 
algorithms~\cite{CCartis_NIMGould_PhLToint_2022}.}

\begin{assumption}
\label{as:C11}
	The function $\hat{J}:u \mapsto J(y(u),u)$ is continuously differentiable 
	in $u$. Moreover, the gradient of $\hat{J}$ with respect to $u$ is 
	$L$-Lipschitz continuous for $L>0$.
\end{assumption}

Note that the first part of Assumption~\ref{as:C11} is implied by 
Assumption~\ref{as:reduced}.

\begin{assumption}
\label{as:bH}
	There exists a positive constant $M_H >0$ such that 
	$\|H_k\| \le M_H$ for all $k$.
\end{assumption}

Assumption~\ref{as:bH} is trivially satisfied when $H_k$ is the zero matrix, 
or whenever the iterates are contained in a compact set. In addition to 
boundedness, we make an additional requirement on that matrix.

\begin{assumption}
\label{as:Hkpsd}
	For any iteration $k$, the matrix $H_k$ is chosen as a positive 
	semidefinite matrix.
\end{assumption}

Note that both the zero matrix and the Gauss-Newton matrix $G_k^\T G_k$ 
are positive semidefinite, and thus satisfy Assumption~\ref{as:Hkpsd}.

\begin{lemma}
\label{le:exactsoldec}
	Let Assumptions \ref{as:C11},\ref{as:bH} and \ref{as:Hkpsd} hold. 
	Suppose that the subproblem~\eqref{eq:lssubpb} is solved exactly at 
	iteration $k$. Then, 
	\begin{equation}
	\label{eq:exactsol}
		s_k = -(H_k+\gamma_k I)^{-1} G_k^\T R_k
	\end{equation}
	where $I$ is the identity matrix in $\R^{d \times d}$. Moreover,
	\begin{equation}
	\label{eq:exactsoldec}
		m_k(u_k) - m_k(u_k+s_k) 
		\ge \frac{1}{2}\frac{\|G_k^\T R_k\|^2}{M_H+\gamma_k}.
	\end{equation}
\end{lemma}

\begin{proof}
	Under Assumption~\ref{as:Hkpsd}, the subproblem~\eqref{eq:lssubpb} is a 
	strongly convex quadratic subproblem. It thus possesses a unique global 
	minimum given by $-(H_k+\gamma_k I)^{-1} G_k^\T R_k$, which is precisely 
	\eqref{eq:exactsol}. Using this formula for $s_k$, we obtain 
	\begin{eqnarray*}
		m_k(u_k) - m_k(u_k+s_k) 
		&\rev{=} &-R_k^\T G_k s_k -\frac{1}{2}s_k^\T (H_k+\gamma_k I) s_k \\
		&= &R_k^\T G_k (H_k+\gamma_k I)^{-1} G_k^\T R_k 
		-\frac{1}{2}R_k^\T G_k (H_k+\gamma_k I)^{-1} G_k^\T R_k \\
		&= &\frac{1}{2}R_k^\T G_k (H_k+\gamma_k I)^{-1} G_k^\T R_k \\
		&\ge &\frac{1}{2}\frac{\|G_k^\T R_k\|^2}{\|H_k+\gamma_k I\|}.
	\end{eqnarray*}
	By Assumption~\ref{as:bH}, we have
	\[
		\rev{\|H_k + \gamma_k I\|} \le \|H_k\|+\gamma_k \le M_H + \gamma_k.
	\]
	Hence, we have
	\[
		m_k(u_k) - m_k(u_k+s_k) \ge 
		\frac{1}{2}\frac{\|G_k^\T R_k\|^2}{M_H+\gamma_k},
	\]
	as required.
\end{proof}

Our second ingredient for a complexity proof consists in bounding the 
value of the regularization parameter. 

\begin{lemma}
\label{le:bdgamma}
	Let Assumptions~\ref{as:C11}, \rev{\ref{as:bH}  and \ref{as:Hkpsd}}
	hold\rev{, and consider the $k$th iteration of Algorithm~\ref{alg:gn}. 
	Suppose that $\|G_k^\T R_k\| >0$.} Then,
	\begin{enumerate}[(i)]
		\item If $k$ is the index of an unsuccessful iteration, then 
		\[
			\rev{
			\gamma_k 
			< \underline{\gamma}:= 
			\tfrac{L+\sqrt{L^2+4(1-\eta)L\,M_H}}{2(1-\eta)}.
			}
		\]
		\item For any iteration $k$, 
		\begin{equation}
		\label{eq:bdgamma}
			\gamma_k 
			\le 
			\gamma_{\max}:= \max\left\{1,\gamma_0,\rev{2\underline{\gamma}}
			 \right\}.
		\end{equation}
	\end{enumerate}
\end{lemma}

\begin{proof}
	Suppose that the $k$th iteration is unsuccessful, i.e. that $\rho_k < \eta$. 
	Then, one has
	\begin{equation}
	\label{eq:rhounsucc}
		\eta(m_k(u_k+s_k)-m_k(u_k)) < \hat{J}(u_k+s_k)-\hat{J}(u_k).
	\end{equation}
	Using Assumption~\ref{as:C11}, a Taylor expansion of $\hat{J}$ around $u_k$ 
	yields
	\begin{eqnarray*}
		\hat{J}(u_k+s_k) - \hat{J}(u_k) 
		&\le &\nabla \hat{J}(u_k)^\T s_k + \frac{L}{2}\|s_k\|^2 \\
		&= &g_k^\T s_k + \frac{L}{2}\|s_k\|^2 \\
		&= &m_k(u_k+s_k)-m_k(u_k)- \rev{\frac{1}{2}s_k^\T H_k s_k 
		-\frac{\gamma_k}{2}\|s_k\|^2 }
		+ \frac{L}{2}\|s_k\|^2  \\
		&\le &m_k(u_k+s_k)-m_k(u_k) \rev{+\frac{L-\gamma_k}{2}\|s_k\|^2},
	\end{eqnarray*}
	where the last inequality holds because of Assumption~\ref{as:Hkpsd}. 
	Combining this inequality with~\eqref{eq:rhounsucc}, we obtain that 
	\begin{eqnarray*}
		\eta(m_k(u_k+s_k)-m_k(u_k)) 
		&< &\hat{J}(u_k+s_k)-\hat{J}(u_k) \\
		\Rightarrow
		\eta(m_k(u_k+s_k)-m_k(u_k)) 
		&< &m_k(u_k+s_k)-m_k(u_k)+\frac{\rev{L-\gamma_k}}{2}\|s_k\|^2 \\
		\Rightarrow
		(1-\eta)(m_k(u_k)-m_k(u_k+s_k)) 
		&< &\frac{\rev{L-\gamma_k}}{2}\|s_k\|^2.
	\end{eqnarray*}
	From Lemma~\ref{le:exactsoldec}, we obtain both an expression for $s_k$ 
	and a bound on the left-hand side. Noting that 
	\rev{
	\[
		\|s_k\| \le \|(H_k+\gamma_k I)^{-1}\| \|G_k^\T R_k\| 
		\le \frac{\|G_k^\T R_k\|}{\gamma_k},
	\]
	}
	we obtain
	\begin{eqnarray*}
		(1-\eta)(m_k(u_k)-m_k(u_k+s_k)) 
		&< &\frac{L}{2}\|s_k\|^2 \\
		\Leftarrow
		\frac{(1-\eta)}{2}\frac{\|G_k^\T R_k\|^2}{M_H+\gamma_k} 
		&< &\frac{\rev{L-\gamma_k}}{2}\frac{\|G_k^\T R_k\|^2}{\rev{\gamma_k^2}} \\
		\rev{\Leftrightarrow
		\frac{1-\eta}{M_H+\gamma_k}} 
		&< &\rev{\frac{L-\gamma_k}{\gamma_k^2}} \\
		\rev{\Leftrightarrow
		(2-\eta)\gamma_k^2 - (L-M_H)\gamma_k -L\,M_H < 0.
		}
	\end{eqnarray*}
	Overall, we have shown that if the $k$th iteration is unsuccessful, then 
	necessarily \rev{$(2-\eta)\gamma_k^2 - (L-M_H)\gamma_k -L\,M_H < 0$, which can 
	only occur as long as
	\[
		\gamma_k 
		< 
		\underline{\gamma}:=
		\frac{L-M_H+\sqrt{(L-M_H)^2+4(2-\eta)L\,M_H}}{2(2-\eta)}.
	\]
	}
	By a contraposition 
	argument, we then obtain that 
	$\gamma_k \ge \rev{\underline{\gamma}}$ 
	implies that the iteration is successful and that 
	$\gamma_{k+1} \le \gamma_k$. Combining this observation with the initial 
	value of $\gamma_0$ and the update mechanism for $\gamma_k$, we find that 
	$\gamma_k$ can never exceed 
	$\max\{\gamma_0,2\underline{\gamma}\} \le \gamma_{\max}$, proving the desired 
	result.
\end{proof}

\rev{Note that we choose $\gamma_{\max}$} to be greater than or equal to $1$ 
in order to simplify our bounds later on, but that the analysis below extends 
to the choice $\gamma_{\max}=\max\{\gamma_0,2\underline{\gamma}\}$. 

We now provide our first iteration complexity bound, that focuses on 
successful iterations.

\begin{lemma}
\label{le:exactsuccits}
	Let Assumptions~\ref{as:C11}, \ref{as:bH} and \ref{as:Hkpsd} hold. Let
	$\epsg \in (0,1)$, and let $\cS_{\epsg,\epsR}$ denote the set of successful 
	iterations for which $u_k$ does not satisfy~\eqref{eq:epspt}. Then, 
	\begin{equation}
	\label{eq:exactsuccits}
		\left| S_{\epsg,\epsR}\right| 
		\quad \le \quad
		\left\lceil
		\mathcal{C}_{\cS} \log(2 \hat{J}(u_0) \epsR^{-2}) \epsR^{-2}
		\right\rceil + 1,
	\end{equation}
	where $\mathcal{C}_{\cS} = \tfrac{M_H+\gamma_{\max}}{\eta}$.
\end{lemma}

\begin{proof}
	Let $k \in \cS_{\epsg,\epsR}$. By definition, the corresponding iterate 
	$u_k$ satisfies
	\begin{equation}
	\label{eq:notepspt}
		\|R_k\| \ge \epsR \quad \mbox{and} \quad 
		\frac{\|G_k^\T R_k\|}{\|R_k\|} \ge \epsg.
	\end{equation}
	Moreover, since $k$ corresponds to a successful iteration, we have 
	$\rho_k \ge \eta$, i.e.
	\[
		\hat{J}(u_k)-\hat{J}(u_{k+1}) \ge \eta\left(m_k(u_k)-m_k(u_k+s_k)\right) 
		\ge \eta \frac{\|G_k^\T R_k\|^2}{2(M_H + \gamma_k)} 
		\ge \eta \frac{\|G_k^\T R_k\|^2}{2(M_H + \gamma_{\max})},
	\]
	where we used the results of Lemmas~\ref{le:exactsoldec} 
	and~\ref{le:bdgamma} to bound the model decrease and $\gamma_k$, 
	respectively. Combining the last inequality with~\eqref{eq:notepspt} leads 
	to
	\begin{eqnarray*}
		\hat{J}(u_k) - \hat{J}(u_{k+1}) 
		&\ge &\frac{\eta}{2(M_H + \gamma_{\max})}\|G_k^\T R_k\|^2 \\
		&= &\frac{\eta}{2(M_H + \gamma_{\max})} 
		\frac{\|G_k^\T R_k\|^2}{\|R_k\|^2} \|R_k\|^2 \\
		&\ge &\frac{\eta}{2(M_H + \gamma_{\max})} \epsg^2 \|R_k\|^2 \\
		&= &\frac{\eta}{M_H+\gamma_{\max}} \epsg^2 \hat{J}(u_k),
	\end{eqnarray*}
	where the last line follows by definition of $\hat{J}(u_k)$. Since 
	$\tfrac{\eta}{M_H+\gamma_{\max}}\epsg^2 \in (0,1)$ by definition of all 
	quantities involved, we obtain that
	\begin{equation}
	\label{eq:decJnotepspt} 
		\left(1-\frac{\eta}{M_H+\gamma_{\max}} \epsg^2 \right) \hat{J}(u_k) 
		\ge \hat{J}(u_{k+1}).
	\end{equation}
	Let now $\cS_{\epsg,\epsR}^k := \{\ell < k | \ell \in \cS_{\epsg,\epsR}\}$.
	Recalling that the iterate only changes on successful iterations and that 
	the function $\hat{J}$ is bounded below by $0$, we obtain that
	\begin{eqnarray*}
		\left(1-\frac{\eta}{M_H+\gamma_{\max}}\epsg^2 \right)^{\left|\cS_{\epsg,\epsR}^k\right|}
		\hat{J}(u_0) &\ge &\hat{J}(u_k) \\
		\left(1-\frac{\eta}{M_H+\gamma_{\max}}\epsg^2 \right)^{\left|\cS_{\epsg,\epsR}^k\right|}
		\hat{J}(u_0) &\ge &\frac{1}{2}\epsR^2,
	\end{eqnarray*}
	where the last line uses $k \in \cS_{\epsg,\epsR}$. Taking logarithms and re-arranging, 
	we arrive at
	\begin{eqnarray*}
		\left|\cS_{\epsg,\epsR}^k\right|\ln\left(1-\frac{\eta}{M_H+\gamma_{\max}}\epsg^2 \right)
		&\ge &\ln\left(\epsR^2 /(2 \hat{J}(u_0))\right) \\
		\left|\cS_{\epsg,\epsR}^k\right|
		&\le &\frac{\ln\left(\epsR^2 /(2 \hat{J}(u_0)\right)}
		{\ln\left(1-\frac{\eta}{M_H+\gamma_{\max}}\epsg^2 \right)} \\
		&\le &\ln\left(2\hat{J}(u_0)\epsR^{-2}\right)\frac{M_H+\gamma_{\max}}{\eta}\epsg^{-2},
	\end{eqnarray*}
	where the last inequality comes from $-\ln(1-t) \ge t$ for any $t \in (0,1)$. 
	As a result, we obtain that
	\[
		\left|\cS_{\epsg,\epsR}\right| 
		\le 1 + \ln\left(2 \hat{J}(u_0)\epsR^{-2}\right)\frac{M_H+\gamma_{\max}}{\eta}\epsg^{-2},
	\]
	where the additional $1$ accounts for the largest iteration in $\cS_{\epsg,\epsR}$.
\end{proof}

\begin{lemma}
\label{le:exactsuccunsucc}
	Under the assumptions of Lemma~\ref{le:exactsuccits}, let $\cU_{\epsg,\epsR}$ 
	be the set of unsuccessful iterations for which~\eqref{eq:epspt} does not hold. Then,
	\begin{equation}
	\label{eq:exactsuccunsucc}
		\left| \cU_{\epsg,\epsR}\right| 
		\le 
		\left\lceil 1 + \log_2 \left(\gamma_{\max}\right) 
		\right\rceil
		\left| \cS_{\epsg,\epsR} \right|.
	\end{equation}
\end{lemma}

\begin{proof}
	The proof tracks that 
	of~\cite[Lemma 2.5]{FECurtis_DPRobinson_CWRoyer_SJWright_2021} for 
	the trust-region case. Between two successful iterations, the value 
	of $\gamma_k$ only increases by factors of $2$. Combining this observation 
	with the fact that $\gamma_k \le \gamma_{\max}$ per Lemma~\ref{le:bdgamma} 
	and accounting for the first successful iteration leads to the final 
	result.
\end{proof}

Combining Lemmas~\ref{le:exactsuccits} and~\ref{le:exactsuccunsucc} finally 
yields our main complexity result.

\begin{theorem}
\label{th:wccexactits}
	Under Assumptions~\ref{as:C11}, \ref{as:bH} and \ref{as:Hkpsd}, the number 
	of successful iterations (and Jacobian evaluations) 
	before reaching an iterate satisfying~\eqref{eq:epspt} satisfies
	\begin{equation}
	\label{eq:wccexactsucc}
		|\cS_{\epsg,\epsR}| = 
		\cO\left( \log(\epsR^{-1}) \epsg^{-2} \right)
	\end{equation}
	and the total number of iterations (and residual evaluations) 
	before reaching such an iterate satisfies
	\begin{equation}
	\label{eq:wccexactits}
		|\cS_{\epsg,\epsR}|+|\cU_{\epsg,\epsR}| 
		= \cO\left( \log(\epsR^{-1}) \epsg^{-2} \right).
	\end{equation}
\end{theorem}

The result of Theorem~\ref{th:wccexactits} improves over that obtained by 
Bergou et al~\cite{EBergou_YDiouane_VKungurtsev_CWRoyer_2022} in a more 
general setting, and is consistent with that in Gould et 
al~\cite{NIMGould_TRees_JAScott_2019}, where a series of results with 
vanishing dependencies in $\epsR$ were established. Compared to those 
latter results, our bounds~\eqref{eq:wccexactsucc} 
and~\eqref{eq:wccexactits} have logarithmic dependency on $\epsR$ but 
do not involve increasingly larger constants.

To end this section, we provide a result tailored to our implicit 
constrained setup, that accounts for the operations that are performed 
throughout the course of the algorithm.

\begin{corollary}
\label{co:wccexactPDE}
	Under the assumptions of Theorem~\ref{th:wccexactits}, suppose further 
	than Assumption~\ref{as:reduced} holds and that Jacobian evaluations are 
	performed using Algorithm~\ref{alg:jacLS}. Then, the number of 
	solves of the implicit constraint for $y$ is 
	\begin{equation}
	\label{eq:wccexactPDEstate}
		1+|\cS_{\epsg,\epsR}|+|\cU_{\epsg,\epsR}|  
		= \cO\left( \log(\epsR^{-1}) \epsg^{-2} \right),
	\end{equation}
	while the number of adjoint solves (using~\eqref{eq:zetaeq})
	is
	\begin{equation}
	\label{eq:wccexactPDEadj}
		1+|\cS_{\epsg,\epsR}| 
		= \cO\left( \log(\epsR^{-1}) \epsg^{-2} \right).
	\end{equation}
\end{corollary}

\subsection{Inexact variants}
\label{subsec:inexact}

We now consider solving the subproblem~\eqref{eq:exactsol} in an inexact 
fashion. Such a procedure is classical in large-scale optimization, and is 
primarily relevant whenever $H_k$ is not chosen as a constant matrix.

\begin{assumption}
\label{as:inexsteps}
	For any iteration $k$, the step $s_k$ is chosen so as to satisfy
	\begin{equation}
	\label{eq:inexsteps}
		(H_k+\gamma_k I)s_k = -g_k + t_k, \quad 
		\|t_k\| \le \theta \sqrt{\frac{\gamma_k}{\|H_k\|+\gamma_k}} \|g_k\|
	\end{equation}
	for $\theta \in [0,1)$.
\end{assumption}

Assuming that the linear system is solved to the accuracy expressed in 
condition~\eqref{eq:inexsteps}, one can establish the following result.

\begin{lemma}
\label{le:inexprop}
	Let Assumptions~\ref{as:bH}, \ref{as:Hkpsd} and \ref{as:inexsteps} hold.
	For any iteration $k$, the step $s_k$ satisfies
	\begin{equation}
	\label{eq:inexnorm}
		\| s_k \|
		\le \frac{(1+\theta)\|G_k^\T R_k\|}{\gamma_k}
	\end{equation}
	and
	\begin{equation}
	\label{eq:inexdec}
		m_k(u_k)-m_k(u_k+s_k) 
		\ge \frac{1-\theta^2}{2}\frac{\|G_k^\T R_k \|^2}{M_H+\gamma_k}.
	\end{equation}
\end{lemma}

\begin{proof}
	Using Assumption~\ref{as:inexsteps} gives
	\[
		\|t_k\| 
		\le \theta \sqrt{\frac{\gamma_k}{\|H_k\| +\gamma_k}}\|G_k^\T R_k\| 
		\le \theta \|G_k^\T R_k\|.
	\]
	Since $s_k = (H_k+\gamma_k I)^{-1}(-g_k+t_k)$ by construction, we obtain
	\begin{eqnarray*}
		\|s_k\| = \left\| (H_k+\gamma_k I)^{-1} (-g_k+t_k) \right\| 
		&\le &\frac{\|g_k\|+\|t_k\|}{\|H_k+\gamma_k I\|}  \\
		&\le &(1+\theta)\frac{\|g_k\|}{\|H_k\|+\gamma_k} \\
		&\le &\frac{(1+\theta)\|G_k^\T R_k\|}{\gamma_k},
	\end{eqnarray*}
	proving~\eqref{eq:inexnorm}.
	
	We now use this inequality together with the definition of $s_k$ to bound the 
	model decrease:
	\begin{eqnarray*}
		m_k(u_k) - m_k(u_k+s_k) 
		&= &-g_k^\T s_k -\frac{1}{2}s_k^\T (H_k+\gamma_k I) s_k \\
		&= &-g_k^\T (H_k+\gamma_k I)^{-1}(-g_k+t_k) 
		-\frac{1}{2}(-g_k+t_k)^\T (H_k+\gamma_k I)^{-1} (-g_k+t_k) \\
		&= &\frac{1}{2} g_k^\T (H_k+\gamma_k I)^{-1} g_k 
		-\frac{1}{2} t_k^\T (H_k+\gamma_k I)^{-1} t_k.
	\end{eqnarray*}
	Using Cauchy-Schwarz inequality, we obtain on one hand
	\[
		 g_k^\T (H_k+\gamma_k I)^{-1} g_k \ge \frac{\|g_k\|^2}{\|H_k+\gamma_k I\|} 
		 \ge \frac{\|g_k\|^2}{\|H_k\|+\gamma_k},
	\]
	while on the other hand
	\[
		t_k^\T (H_k+\gamma_k I)^{-1} t_k \le \rev{\|(H_k+\gamma_k I)^{-1}\|} \|t_k\|^2
		\le \frac{\|t_k\|^2}{\gamma_k} \le \frac{\theta^2 \|g_k\|^2}{\|H_k\|+\gamma_k},
	\]
	where the last inequality comes from Assumption~\ref{as:inexsteps}. 
	As a result, we arrive at
	\begin{eqnarray*}
		m_k(u_k) - m_k(u_k+s_k) 
		&\ge &\frac{1}{2} \frac{\|g_k\|^2}{\|H_k\|+\gamma_k} 
		-\frac{\theta^2}{2} \frac{\|g_k\|^2}{\|H_k\|+\gamma_k} \\
		&= &\frac{1-\theta^2}{2}\frac{\|g_k\|^2}{\|H_k\|+\gamma_k} \\
		&\ge &\frac{1-\theta^2}{2}\frac{\|g_k\|^2}{M_H+\gamma_k},
	\end{eqnarray*}	
	using Assumption~\ref{as:bH} to bound $\|H_k\|$. This 
	proves~\eqref{eq:inexdec}.
\end{proof}

Similarly to the exact case, we now prove that the regularization parameter 
is bounded from above.

\begin{lemma}
\label{le:inexbdgamma}
	Let Assumptions~\ref{as:C11}, \ref{as:bH}, \ref{as:Hkpsd} and 
	\ref{as:inexsteps}  hold. \rev{Consider the $k$th iteration of 
	Algorithm~\ref{alg:gn}, and suppose that $\|G_k^\T R_k\|>0$.} Then,
	\begin{enumerate}[(i)]
		\item If $k$ is the index of an unsuccessful iteration, then 
		$\gamma_k < \tfrac{L (1+\theta)^2}{(1-\eta)(1-\theta^2)}$.
		\item For any iteration $k$, 
		\begin{equation}
		\label{eq:inexbdgamma}
			\gamma_k 
			\le 
			\gamma_{\max}^{in}:= \max\left\{1,\gamma_0,
			\frac{L (1+\theta)^2}{(1-\eta)(1-\theta^2)}\right\}.
		\end{equation}
	\end{enumerate}
\end{lemma}

\begin{proof}
	By the same reasoning as in the proof of Lemma~\ref{le:bdgamma}, we 
	know that for any unsuccessful iteration, we have
	\begin{equation}
	\label{eq:inexunsucc}
		(1-\eta)(m_k(u_k)-m_k(u_k+s_k)) 
		< \frac{\rev{L-\gamma_k}}{2}\|s_k\|^2.
	\end{equation}
	Using now the properties~\eqref{eq:inexnorm} and~\eqref{eq:inexdec} 
	in~\eqref{eq:inexunsucc}, we obtain:
	\begin{eqnarray*}
		(1-\eta)(m_k(u_k)-m_k(u_k+s_k)) 
		&< &\frac{\rev{L-\gamma_k}}{2}\|s_k\|^2 \\
		\Leftarrow
		\frac{(1-\eta)(1-\theta^2)}{2}\frac{\|G_k^\T R_k\|^2}{M_H+\gamma_k} 
		&< &\frac{(1+\theta)^2\rev{(L-\gamma_k)}}{2}\frac{\|G_k^\T R_k\|^2}{\rev{\gamma_k^2}} \\
		\rev{
		\Leftrightarrow
		\frac{(1-\eta)(1-\theta^2)}{M_H+\gamma_k}} 
		&\rev{<} &\rev{\frac{(1+\theta)^2(L-\gamma_k)}{\gamma_k^2}} \\
		\rev{
		\Leftrightarrow
		(2(1+\theta)-\eta)\gamma_k^2 
		-\left[(1+\theta)^2 L-M_H\right]\gamma_k-(1+\theta)^2 L} 
		&\rev{<} &\rev{0}.
	\end{eqnarray*}
	Overall, we have shown that if the $k$th iteration is unsuccessful, then 
	necessarily 
	\[
		\rev{\gamma_k < \underline{\gamma}^{in}:= 
		\frac{(1+\theta)^2 L -M_H + \sqrt{((1+\theta)^2 L -M_H)^2 
		+ 4 (1+\theta)^2(2(1+\theta)-\eta)L}}{2(2(1+\theta)-\eta)}.}
	\] 
	\rev{Using} the updating rules on $\gamma_k$ and accounting for $\gamma_0$ 
	we obtain that \rev{$\gamma_k \le \max\left\{\gamma_0,
	\underline{\gamma}^{in} \right\} \le \gamma_{\max}^{in}$} 
	for all $k$, proving the desired result.
\end{proof}

We can now state an iteration complexity result for the inexact variant.

\begin{lemma}
\label{le:inexsuccits}
	Let Assumptions~\ref{as:C11}, \ref{as:bH}, \ref{as:Hkpsd} and 
	\ref{as:inexsteps} hold. Let
	$\epsg \in (0,1)$, and let $\cS_{\epsg,\epsR}^{in}$ denote the set of successful 
	iterations for which $u_k$ does not satisfy~\eqref{eq:epspt}. Then, 
	\begin{equation}
	\label{eq:inexsuccits}
		\left| \cS_{\epsg,\epsR}^{in}\right| 
		\quad \le \quad
		\left\lceil
		\mathcal{C}_{\cS}^{in} \log(2 \hat{J}(u_0)\,\epsR^{-2})\epsg^{-2}
		\right\rceil + 1,
	\end{equation}
	where $\mathcal{C}_{\cS}^{in} = 
	\tfrac{M_H + \gamma_{\max}^{in}}{\eta(1-\theta^2)}$.
\end{lemma}

\begin{proof}
	Let $k \in \cS_{\epsg,\epsR}$. By definition, the $k$th iteration 
	is successful, and we have per Lemma~\ref{le:inexbdgamma}
	\[
		\hat{J}(u_k)-\hat{J}(u_k+s_k) \ge \eta\left(m_k(u_k)-m_k(u_k+s_k)\right) 
		\ge \frac{\eta(1-\theta^2)}{2} \frac{\|G_k^\T R_k\|^2}{M_H + \gamma_k} 
		\ge \frac{\eta(1-\theta^2)}{2} \frac{\|G_k^\T R_k\|^2}
		{M_H + \gamma_{\max}^{in}},
	\]
	where the last inequality is a consequence of Lemma~\ref{le:inexbdgamma}.
	In addition, the corresponding iterate $u_k$ satisfies~\eqref{eq:notepspt}, 
	leading to 
	\begin{eqnarray*}
		\hat{J}(u_k) - \hat{J}(u_k+s_k) 
		&\ge &\frac{\eta(1-\theta^2)}{2} \frac{\|G_k^\T R_k\|^2}
		{M_H + \gamma_{\max}^{in}} \\
		&= & \frac{\eta(1-\theta^2)}{2(M_H + \gamma_{\max}^{in})} 
		\frac{\|G_k^\T R_k\|^2}{\|R_k\|^2} \|R_k\|^2 \\
		&= &\frac{\eta(1-\theta^2)}{M_H + \gamma_{\max}^{in}} 
		\frac{\|G_k^\T R_k\|^2}{\|R_k\|^2} J(u_k)\\
		&\ge &\frac{\eta(1-\theta^2)}{M_H + \gamma_{\max}^{in}} 
		\epsg^2 \hat{J}(u_k).
	\end{eqnarray*}	
	Using that $\frac{\eta(1-\theta^2)}{M_H + \gamma_{\max}^{in}} <1$ then 
	leads to
	\begin{equation}
	\label{eq:decJnotepsptinex} 
		\left(1-\frac{\eta(1-\theta^2)}{M_H+\gamma_{\max}} \epsg^2 \right) \hat{J}(u_k) 
		\ge \hat{J}(u_{k+1}).
	\end{equation}
	By proceeding as in the proof of Lemma~\ref{le:exactsuccits} and 
	using~\eqref{eq:decJnotepsptinex} in lieu of~\eqref{eq:decJnotepsptinex}, 
	one establishes that
	\[
		\left|\cS_{\epsg,\epsR}^{in}\right| 
		\le 1 + \ln\left(2 \hat{J}(u_0)\epsR^{-2}\right)\frac{M_H+\gamma_{\max}}{\eta(1-\theta^2)}
		\epsg^{-2},
	\]
	proving the desired result.
\end{proof}

To connect the number of unsuccessful iterations with that of successful 
iterations, we use the same argument as in the exact case by replacing 
the bound~\eqref{eq:bdgamma} with~\eqref{eq:inexbdgamma}.

\begin{lemma}
\label{le:inexsuccunsucc}
	Under the assumptions of Lemma~\ref{le:inexsuccits}, let 
	$\cU_{\epsg,\epsR}^{in}$ be the set of unsuccessful iterations for which
	\eqref{eq:epspt} does not hold. Then,
	\begin{equation}
	\label{eq:inexsuccunsucc}
		\left| \cU_{\epsg,\epsR}^{in}\right| 
		\le 
		\left\lceil 1 + \log_2 \left(\gamma_{\max}^{in}\right) 
		\right\rceil
		\left| \cS_{\epsg,\epsR}^{in} \right|.
	\end{equation}
\end{lemma}

Our next theorem gives the total iteration complexity result by 
combining Lemmas~\ref{le:inexsuccits} and~\ref{le:inexsuccunsucc}.

\begin{theorem}
\label{th:wccinexits}
	Under Assumptions~\ref{as:C11}, \ref{as:bH}, \ref{as:Hkpsd} and 
	\ref{as:inexsteps}, the number 
	of successful iterations (and inexact step calculations) 
	before reaching an iterate satisfying~\eqref{eq:epspt} satisfies
	\begin{equation}
	\label{eq:wccinexsucc}
		|\cS_{\epsg,\epsR}^{in}| = 
		\cO\left( \frac{1}{(1-\theta^2)^2} 
		\log(\epsR^{-1}) \epsg^{-2} \right)
	\end{equation}
	and the total number of iterations (and residual evaluations) 
	before reaching such an iterate satisfies
	\begin{equation}
	\label{eq:wccinexits}
		|\cS_{\epsg,\epsR}^{in}|+|\cU_{\epsg,\epsR}^{in}| 
		= \cO\left( 
		\frac{1}{(1-\theta^2)^2}		
		\log(\epsR^{-1}) \epsg^{-2} \right).
	\end{equation}
\end{theorem}

The results of Theorem~\ref{th:wccinexits} match that of 
Theorem~\ref{th:wccexactits} in terms of dependencies on 
$\epsg$ and $\epsR$. To emphasize the use of inexact steps, we 
highlighted the dependency with respect to the inexact tolerance 
$\theta$. As expected, one notes that this dependency vanishes when 
$\theta=0$ (i.e. when we consider exact steps as in 
Section~\ref{subsec:iter}), and that the complexity bounds worsen 
as $\theta$ gets closer to $1$. A similar observation holds for 
the results in the next corollary, that is a counterpart to 
Corollary~\ref{co:wccexactPDE}.

\begin{corollary}
\label{co:wccinexPDE}
	Under the assumptions of Theorem~\ref{th:wccinexits} as well as 
	Assumption~\ref{as:reduced}, the number of 
	 solves for $y$ is 
	\begin{equation}
	\label{eq:wccinexPDEstate}
		1+|\cS_{\epsg,\epsR}^{in}|+|\cU_{\epsg,\epsR}^{in}|  
		= \cO\left( \log\left(\frac{1}{1-\theta^2}\right) 
		\frac{1}{(1-\theta^2)^2}\, \log(\epsR^{-1}) \epsg^{-2}\right),
	\end{equation}
	while the number of adjoint  solves (using Algorithm~\ref{alg:jacLS}) 
	is
	\begin{equation}
	\label{eq:wccinexPDEadj}
		1+|\cS_{\epsg,\epsR}^{in}|  
		= \cO\left( \frac{1}{(1-\theta^2)^2}\, 
		\log(\epsR^{-1}) \epsg^{-2} \right).
	\end{equation}
\end{corollary}

In addition to the previous results, we can also exploit the inexact nature of 
the steps to provide more precise guarantees on the computation cost of an iteration. More 
precisely, suppose that we apply an iterative solver to the system 
$(H_k+\gamma_k I) s = -g_k$ in order to find an approximate solution 
satisfying Assumption~\ref{as:inexsteps}. In particular, one can resort to 
iterative linear algebra techniques such as Conjugate Gradient (CG), and obtain 
guarantees on the number of matrix-vector products 
necessary to reach the desired accuracy~\cite{JNocedal_SJWright_2006}. 
A result tailored to our setting is presented below.

\begin{proposition}
\label{prop:inexCGk}
	Let Assumption~\ref{as:Hkpsd} hold. Suppose that we 
	apply conjugate gradient (CG) to the linear system $(H_k+\gamma_k I) s = -g_k$, 
	where $g_k,H_k,\gamma_k$ are obtained from the $k$th iteration of 
	Algorithm~\ref{alg:gn}. Then, the conjugate gradient method computes an 
	iterate satisfying~\eqref{eq:inexsteps} after at most
	\begin{equation}
	\label{eq:inexCGk}
		\min\left\{ n, \frac{1}{2}\sqrt{\kappa_k} \log\left(\frac{2\kappa_k}{\theta}
		 \right) \right\}
	\end{equation}
	iterations or, equivalently, matrix-vector products, where 
	$\kappa_k=\tfrac{\|H_k\|+\gamma_k}{\gamma_k}$.
\end{proposition}

\begin{proof}
	Let $s^{(q)}$ be the iterate obtained after applying $q$ iterations 
	of conjugate gradient to $(H_k+\gamma_k I) s = -g_k$. If $q=n$, then 
	necessarily the linear system has been solved exactly and ~\eqref{eq:inexsteps} 
	is trivially satisfied. Thus we assume in what follows that $q<n$.
	
	Standard CG theory gives~\cite[Proof of Lemma 11]{CWRoyer_SJWright_2018}:
	\begin{equation}
	\label{eq:CGerrortrue}
		\|(H_k+\gamma_k I) s^{(q)} + g_k\| 
		\le 2 \sqrt{c_k} \left(\frac{\sqrt{c_k}-1}{\sqrt{c_k}+1}\right)^{q}\|g_k\|,
	\end{equation}
	where $c_k$ is the condition number of $H_k+\gamma_k I$. Noticing that 
	$c_k\le \kappa_k$, we see that~\eqref{eq:CGerrortrue} implies
	\begin{equation}
	\label{eq:CGerror}
		\|(H_k+\gamma_k I) s^{(q)} + g_k\| 
		\le 2 \sqrt{\kappa_k} 
		\left(\frac{\sqrt{\kappa_k}-1}{\sqrt{\kappa_k}+1}\right)^{q}\|g_k\|.
	\end{equation}	
	Suppose now that $s^{(q)}$ does not satisfy~\eqref{eq:inexsteps}. Then,
	\begin{equation}
	\label{eq:notgoodCGstep}
		\|(H_k+\gamma_k I) s^{(q)} + g_k\| 
		\ge \theta \sqrt{\frac{\gamma_k}{\|H_k\|+\gamma_k}} \|g_k\|
		= \frac{\theta}{\sqrt{\kappa_k}} \|g_k\|.
	\end{equation}
	Combining~\eqref{eq:CGerror} and~\eqref{eq:notgoodCGstep} yields
	\begin{eqnarray*}
		\frac{\theta}{\sqrt{\kappa_k}} \|g_k\| 
		&\le &2 \sqrt{\kappa_k} 
		\left(\frac{\sqrt{\kappa_k}-1}{\sqrt{\kappa_k}+1}\right)^{q}\|g_k\| \\
		\frac{\theta}{2 \rev{\kappa_k}} 
		&\le &\left(\frac{\sqrt{\kappa_k}-1}{\sqrt{\kappa_k}+1}\right)^{q}.
	\end{eqnarray*}
	Taking logarithms and rearranging, we arrive at
	\begin{equation}
	\label{eq:boundqnotn}
		q 
		\le 
		\frac{\ln(\theta/(2\kappa_k))}
		{\ln\left(\frac{\sqrt{\kappa_k}-1}{\sqrt{\kappa_k}+1}\right)} 
		\le 
		\frac{\ln(2\kappa_k/\theta)}
		{\ln\left(1+\tfrac{2}{\sqrt{\kappa_k}-1}\right)} 
		\le 
		\frac{1}{2}\sqrt{\kappa_k} \ln\left(\frac{2\kappa_k}{\theta}\right),
	\end{equation}
	where the last inequality used $\ln(1+\tfrac{1}{t}) \ge \tfrac{1}{t+1/2}$. 
	Combining~\eqref{eq:boundqnotn} with the fact that $q \le n$ yields our 
	desired bound.
\end{proof}

Using the bounds on $\gamma_k$ and $\|H_k\|$ from our complexity analysis, we 
see that the value~\eqref{eq:inexCGk} can be bounded from above by
\begin{equation}
\label{eq:inexCG}
	\min\left\{ n, \frac{1}{2}\sqrt{\kappa} 
	\log\left(\frac{2\kappa}{\theta} \right) \right\},
\end{equation}
with $\kappa = \tfrac{M_H+\gamma_{\max}^{in}}{\gamma_{\min}}$. 
Using~\eqref{eq:inexCG} in conjunction with the complexity bound of 
Theorem~\ref{th:wccinexits}, we derive the following bound on the 
number of matrix-vector products.

\begin{corollary}
\label{co:wccinexhvec}
	Under the assumptions of Theorem~\ref{th:wccinexits}, suppose that 
	we apply conjugate gradient to compute inexact steps in 
	Algorithm~\ref{alg:gn}. Then, the algorithm reaches a point 
	satisfying~\eqref{eq:epspt} in at most
	\begin{equation}
	\label{eq:wccinexhvec}
		\begin{array}{ll}
			&\min\left\{n,\frac{1}{2}\sqrt{\kappa}
			\log\left(\frac{2\kappa}{\theta} \right) \right\} 
			\times \left(1+|\cS_{\epsg,\epsR}^{in}| \right) \\
			= &\cO\left( 
			\min\left\{n,\sqrt{\kappa}\log\left(\frac{\kappa}{\theta} \right)
			\right\}\,		
			\log\left(\frac{1}{1-\theta^2}\right) 
			\frac{1}{(1-\theta^2)^2}\,
			\log(\epsR^{-1}) \epsg^{-2} \right)
		\end{array}
	\end{equation}
	matrix-vector products.
\end{corollary}

As a final note, we point out that there exist variants of the conjugate 
gradient method that take advantage of a Gauss-Newton approximation 
$H_k=G_k^\T G_k$\rev{~\cite{CCPaige_MASaunders_1982,JNocedal_SJWright_2006}, 
and require two Jacobian-vector products per iteration instead of a full 
matrix-vector product with $H_k$. The corresponding complexity 
bound is worse than that of Corollary~\ref{co:wccinexhvec} by a factor of $2$,
but the cost of Jacobian-vector products can be much lower than that of 
matrix-vector products. Our implementation from the next section relies on this 
approach.}

\section{Numerical illustration}
\label{sec:num}

In this section, we illustrate the performance of several instances of our 
framework on classical  PDE-constrained optimization problems. Our goal is 
primarily to investigate the practical relevance of using 
condition~\eqref{eq:epspt} as a stopping criterion.
\rev{For this reason, we are mainly interested in the evaluation and 
iteration cost of our algorithm, and therefore report those statistics in the 
rest of the section.}

\rev{We implemented an inexact Gauss-Newton method based on 
Algorithm~\ref{alg:gn} in MATLAB R2023a. The method uses $\eta=0.1$, 
$\gamma_{\min}=10^{-10}$ and 
$\gamma_0=\max\{1,\|g_0\|,\|u_0\|_{\infty}+1\}$. Steps are computed inexactly 
using the conjugate gradient method so as to satisfy 
condition~\eqref{eq:inexsteps} from Assumption~\ref{as:inexsteps}. We compare 
five variants of our algorithm corresponding to 
$$
	\theta \in \{0,\ 10^{-6},\ 10^{-4},\ 10^{-2},\ 10^{-1}, 5.10^{-1}\}.
$$
The case $\theta=0$ corresponds to an exact variant, although the step is 
computed via an iterative linear algebra solver. More precisely, we use the 
\texttt{lsqr} function in MATLAB, which is mathematically 
equivalent to conjugate gradient applied to the subproblem~\eqref{eq:lssubpb} 
and has better numerical stability~\cite{CCPaige_MASaunders_1982}.}

Runs were completed on HP EliteBook x360 1040 G8 Notebook PC with 32Go RAM and 
8 cores 11th Gen Intel Core i7-1165G7 @ 2.80GHz.

\subsection{Elliptic PDE-constrained problem}
\label{subsec:elliptic}

We first consider a standard elliptic optimal control problem, where the 
control is chosen so that the temperature distribution (the state) matches 
a desired distribution as closely as possible~\cite{ESW14,Fredi_2010}. The 
resulting problem can be written as
\begin{equation}
\label{J1}
	\min_{y,u}J(y,u):= \frac{1}{2} \int_{\mathcal{D}} 
		\left[\left( y(u(x))-z(x)\right)^2 
		+ \lambda u(x)^2\right]\,dx,
\end{equation}
\begin{alignat}{2}
\label{J1a}
 \st \,\,\, \quad   -\nabla\cdot(a(x)\nabla y(x)) &= u(x),
 & &  \;\;\;\mbox{in}\;\;  \mathcal{D},\\\nonumber
 y(x) &= \rev{f}, & & \;\;\; \mbox{on}\;\; \partial\mathcal{D},
\end{alignat}
\rev{where  the desired state $z$,  the regularization parameter $\lambda>0$, 
and the function $f\in L^{\infty}(\mathcal{D})$  are all given. 
We set $\mathcal{D}:=[0,1]^2$, $\lambda=0.001$ and 
$a(x)\equiv 1, \;\forall x\in \mathcal{D}$.}
By discretizing~\eqref{J1} and~\eqref{J1a} using piecewise linear finite 
elements on a triangular grid, we arrive at the discretized formulation
\begin{equation*}
\label{discretecost}
	\begin{array}{ll}
	\min_{{\bf y}, {\bf u}} 
		&\frac{1}{2}({\bf y} -{\bf z})^{T} M ({\bf y} -{\bf z})  +\frac{\lambda}{2}{\bf u}^{T} M {\bf u},\\
		& \\
	\mbox{subject\ to} & K{\bf y} = M{\bf u} + \rev{{\bf f}},
	\end{array}
\end{equation*}
where the vectors ${\bf y}, {\bf z}, {\bf u} $ denote the discrete forms of 
the state, the desired state, respectively, and the control variables. 
\rev{The matrices $K$ and $M$ correspond to the stiffness matrix and mass 
matrix, respectively, while the vector ${\bf{f}}$ corresponds to the Dirichlet 
boundary conditions, respectively~\cite[Chapter 5]{ESW14}. Note that, in this 
example, we have set $f\equiv 0$.}

\rev{The} cost function~\eqref{discretecost} 
can be written as $\frac{1}{2}\|R({\bf y},{\bf u})\|^2$ with
\begin{equation}
\label{ellip}
	R({\bf y},{\bf u}) = 
	\begin{bmatrix}
		 M^{1/2} ({\bf y} - {\bf z}) \\
		\sqrt{\beta}M^{1/2} {\bf u}
	\end{bmatrix}, 
\end{equation}
fitting our formulation of interest~\eqref{eq:genpb}. \rev{We adopt 
this approach and apply Algorithm~\ref{alg:gn} to the 
formulation~\eqref{discretecost} with $n=1829$ (see, e.g., \cite{HAntil_SDolgov_AOnwunta_2023}). The 
linear PDE is solved using MATLAB's \texttt{pcg} method with the HSL\_MI20 
algebraic multigrid preconditioner~\cite{HSL_2023}. We use ${\bf u}_0 =  {\bf 1}$ 
and present results for two possible values of the desired state ${\bf z}$. For both 
choices, we report the number of PDE solves and Jacobian-vector products 
required by the methods to reach a point satisfying~\eqref{eq:epspt} within 
a computational budget of 300 iterations
\footnote{\rev{The algorithms stopped making progress after 300 iterations, thus 
we used this value as a baseline for our experiments.}}.
}

\rev{In Tables~\ref{tab:ellipticz0pde}--\ref{tab:ellipticz0jvp}, we 
provide the results for ${\bf z}={\bf 0}$. Note that in this case, case 
${\bf u}={\bf 0}$ gives a zero residual and 
the problem has a zero residual solution. Except in one case identified with 
a dash\footnote{\rev{In all tables, a dash indicates that the stopping condition was 
not satisfied within the computational budget.}}, all variants succeed in finding a 
point with small residual (in the sense of the tolerance $\epsR$) but never 
decrease the scaled gradient norm below the required tolerance $10^{-4}$. We 
observe that the number of PDE solves (or, equivalently, the number of iterations 
of Algorithm~\ref{alg:gn}) increases mildly as $\epsR$ decreases, as predicted by 
our complexity bounds that depend logarithmically on $\epsR$. Similarly, the choice 
for $\theta$ has a mild impact on the number of PDE solves, except when $\theta$ is chosen 
relatively close to $1$. In terms of Jacobian-vector products, however, the 
influence of $\theta$ is more noticeable, especially for the smallest and largest 
values for this parameter. Overall, these results both confirm the 
interest of our criterion~\eqref{eq:epspt} in presence of small 
residuals and the impact of inexactness on the evaluation complexity 
of the method.}

\begin{table}[ht!]
	\rev{
	\begin{center}
	\begin{tabular}{llllllll}\toprule
		$\epsR$ &$\epsg$ &\multicolumn{6}{c}{$\theta$} \\ \cmidrule{3-8}
		& &$0$ &$10^{-6}$ &$10^{-4}$ &$10^{-2}$ &$10^{-1}$ &$5.10^{-1}$\\ \hline
%
		$10^{-3}$ &$10^{-4}$ &26 &26 &26 &26 &26 &39 \\
		$10^{-6}$ &$10^{-4}$ &31 &31 &31 &31 &32 &231 \\
		$10^{-9}$ &$10^{-4}$ &33 &33 &33 &33 &34 &- \\
		\bottomrule
	\end{tabular}
	\end{center}
	\vspace*{0.5ex}
	\caption{\rev{Number of PDE solves for six variants of Algorithm~\ref{alg:gn} on 
	the elliptic PDE problem~\eqref{discretecost} using ${\bf z}={\bf 0}$ as 
	desired state.}}
	\label{tab:ellipticz0pde}
	}
\end{table}

\begin{table}[ht!]
	\rev{
	\begin{center}
	\begin{tabular}{llllllll}\toprule
		$\epsR$ &$\epsg$ &\multicolumn{6}{c}{$\theta$} \\ \cmidrule{3-8}
		& &$0$ &$10^{-6}$ &$10^{-4}$ &$10^{-2}$ &$10^{-1}$ &$5.10^{-1}$ \\ \hline
%
		$10^{-5}$ &$10^{-4}$ &441 &243 &173 &99  &57  &48\\
		$10^{-7}$ &$10^{-4}$ &646 &448 &370 &206 &133 &150\\
		$10^{-9}$ &$10^{-4}$ &728 &530 &452 &264 &177 &-\\
		\bottomrule
	\end{tabular}
	\end{center}
	\vspace*{0.5ex}
	\caption{\rev{Jacobian-vector products for six variants of Algorithm~\ref{alg:gn} on 
	the elliptic PDE problem~\eqref{discretecost} using ${\bf z}={\bf 0}$ as 
	desired state.}}
	\label{tab:ellipticz0jvp}
	}
\end{table}

\rev{We report the results for ${\bf z}={\bf 1}$ in 
Tables~\ref{tab:ellipticz1pde}--\ref{tab:ellipticz1jvp}. The problem now 
possesses large residuals. As a result, the scaled gradient 
condition is a better stopping criterion, and it is indeed triggered up 
to tolerance $\epsg=10^{-8}$ for an exact variant of the algorithm ($\theta=0$), 
even though the residual has not decreased below $10^{-2}$. Interestingly, we 
observe that both quantities of interest exhibit a worse dependency on $\epsg$ 
than on $\epsR$ in the previous tables, in that the number of iterations 
increases more rapidly as $\epsg$ decreases. Inexactness of the steps seems to 
increase the cost of reaching accuracies below $10^{-6}$ (note that the methods 
still reach a scaled gradient below $3.10^{-6}$ within $300$ iterations). The 
two observations are consistent with Corollary~\ref{co:wccinexPDE} and 
Corollary~\ref{co:wccinexhvec}, respectively.}

\begin{table}[ht!]
	\rev{
	\begin{center}
	\begin{tabular}{llllllll}\toprule
		$\epsR$ &$\epsg$ &\multicolumn{6}{c}{$\theta$} \\ \cmidrule{3-8}
		& &$0$ &$10^{-6}$ &$10^{-4}$ &$10^{-2}$ &$10^{-1}$ &$5.10^{-1}$\\ \hline
		$10^{-2}$ &$10^{-4}$ &24 &24 &24 &24 &24 &23 \\
		$10^{-2}$ &$10^{-6}$ &28 &28 &28 &- &- &- \\
		$10^{-2}$ &$10^{-8}$ &30 &30 &-  &- &- &- \\
		\bottomrule
	\end{tabular}
	\end{center}
	\vspace*{0.5ex}
	\caption{\rev{Number of PDE solves for six variants of Algorithm~\ref{alg:gn} on 
	the elliptic PDE problem~\eqref{discretecost} using ${\bf z}={\bf 1}$ as 
	desired state.}}
	\label{tab:ellipticz1pde}
	}
\end{table}

\begin{table}[ht!]
	\rev{
	\begin{center}
	\begin{tabular}{llllllll}\toprule
		$\epsR$ &$\epsg$ &\multicolumn{6}{c}{$\theta$} \\ \cmidrule{3-8}
		& &$0.00$ &$10^{-6}$ &$10^{-4}$ &$10^{-2}$ &$10^{-1}$ &$5.10^{-1}$ \\ \hline
		$10^{-2}$ &$10^{-4}$ &357 &157 &111 &59 &39 &22\\
		$10^{-2}$ &$10^{-6}$ &521 &281 &173 &- &- &-\\
		$10^{-2}$ &$10^{-8}$ &603 &331 &-   &- &- &-\\
		\bottomrule
	\end{tabular}
	\end{center}
	\vspace*{0.5ex}
	\caption{\rev{Jacobian-vector products for six variants of Algorithm~\ref{alg:gn} on 
	the elliptic PDE problem~\eqref{discretecost} using ${\bf z}={\bf 1}$ as 
	desired state.}}
	\label{tab:ellipticz1jvp}
	}
\end{table}

\subsection{Burgers' equation}
\label{subsec:burgers}

We now describe our second test problem based on Burgers' equation, a 
simplified model for turbulence~\cite{MMBaumann_2013,Reyes_Kunish_2004,
Fredi_Volkwein_2001,MHeinkenschloss_2018}. Control problems of this form are 
often considered as the most fundamental nonlinear problem to handle. In our 
case, they illustrate the performance of our algorithms in a 
nonlinear, implicitly constrained setting.

Our formulation is as follows:
\begin{equation}
\label{eq:controlpb}
	\left\{
	\begin{array}{lll}
		\min_{y,u}J(y,u):=
		&\frac{1}{2} \int_{0}^T \int_{0}^L 
		\left[\left( y(t,x)-z(t,x)\right)^2 
		+ \omega u(t,x)^2\right]\,dt\,dx
		& \\
		\st 
		& y_t  + \frac{1}{2}\left(y^2 + \nu y_x\right)_x  = f + u 
		&(x,t) \in (0,L) \times (0,T) \\
		&y(t,0) = y(t,L) = 0 
		&t \in (0,T) \\
		&y(0,x) = y_0(x) 
		&x \in (0,L).
	\end{array}
	\right.
\end{equation}
Here $L$ and $T$ are space and time horizons, respectively; 
$u: [0,T] \times [0,L] \rightarrow \R$ is the control of our optimization 
problem; $y: [0,T] \times [0,L] \rightarrow \R$ is the state; 
$z: [0,T] \times [0,L] \rightarrow \R$ is the desired state; $\omega>0$ 
is a regularization parameter; $f$ is a source term, and $\nu$ is 
the viscosity parameter. 

Given $u$, $y$ can be computed by solving the PDE 
\begin{equation}
\label{eq:burgerpde}
	\begin{array}{lll}
		y_t  + \frac{1}{2}\left(y^2 + \nu y_x\right)_x  &= f + u 
		&(x,t) \in (0,L) \times (0,T) \\
		y(t,0) = y(t,L) &= 0 
		&t \in (0,T) \\
		y(0,x) &= y_0(x) 
		&x \in (0,L).
	\end{array}
\end{equation}

We discretize~\eqref{eq:burgerpde} in time by applying the backward Euler scheme to 
Burgers' equation and a rectangle rule for the discretization of the objective 
function, while the spatial variable is approximated by piecewise linear 
finite elements. As a result, we obtain the following discretized version of 
problem~\eqref{eq:controlpb}:
\begin{equation}
\label{eq:discretepb}
	\left\{
	\begin{array}{ll}
		\minimize_{u_0,\dots,u_{N_t} \in \R^{N_x}} 
		&J(y_0,\dots,y_{N_t},u_0,\dots,u_{N_t}) 
		\\
		\st 
		& c_{i+1}(y_i,y_{i+1},u_{i+1};\nu)=0, \quad i=0,\dots,N_t-1, \\
	\end{array}
	\right.
\end{equation}
where 
\begin{equation}
\label{eq:discreteJ}
	J(y_0,\dots,y_{N_t},u_0,\dots,u_{N_t}) 
	:= \delta_t\sum_{i=0}^{N_t} \left(\tfrac{1}{2}(y_i-z)^\T M (y_i-z) 
	+ \tfrac{\omega}{2} u_i^\T M u_i \right)
\end{equation}
and
\begin{equation}
\label{eq:discretec}
	c_{i+1}(y_i,y_{i+1},u_{i+1};\nu) = 
	\frac{1}{\delta_t} M y_{i+1} - \frac{1}{\delta_t} M y_i 
	+ \frac{1}{2} B y_{i+1} \odot y_{i+1} + \nu C y_{i+1} - f - M u_{i+1},
\end{equation}
and $\odot$ denotes the entrywise product.
In those equations, $\delta_t = \frac{T}{N_t}$ represents the time step of 
the discretization, while $M,B,C,\{f_i\}$ are discretized versions of the 
operators and the source term arising from the continuous formulation. More 
precisely, we have
\[
	M = \frac{h}{6} 
	\begin{bmatrix}
	4 &1 & & & \\
	1 &4 &1 & & \\
	 &\ddots &\ddots &\ddots & \\
	 & &1 &4 &1 \\
	& & &1 &4
	\end{bmatrix} \in \R^{N_x \times N_x},
	\quad
	B = 
	\begin{bmatrix}
	0 &1/2 & & & \\
	-1/2 &0 &1/2 & & \\
	 &\ddots &\ddots &\ddots & \\
	 & &-1/2 &0 &1/2 \\
	& & &-1/2 &0
	\end{bmatrix}
	\in \R^{N_x \times N_x}
\]
and
\[
	C = 
	\frac{1}{h}
	\begin{bmatrix}
	2 &-1 & & & \\
	-1 &2 &-1 & & \\
	 &\ddots &\ddots &\ddots & \\
	 & &-1 &2 &-1 \\
	& & &-1 &2
	\end{bmatrix}
	\in \R^{N_x \times N_x},
	\quad
	f = 
	\begin{bmatrix}
	f_1 \\
	\vdots \\
	f_{N_x}
	\end{bmatrix}
	\in \R^{N_x},
\]
with $h=\tfrac{L}{N_x}$ being the space discretization step.
Following previous work~\cite{MMBaumann_2013,MHeinkenschloss_2018}, we 
assume that the desired state $z$ does not depend on time. 

To reduce the effects of boundary layers, we discretize Burgers' equation 
using continuous piecewise linear finite elements built on a piecewise uniform 
mesh. We then solve the resulting discretized nonlinear PDE at each time step 
using Newton's method~\cite{MMBaumann_2013}.

Letting ${\bf u}$ (resp. ${\bf y}$) as the concatenation of $u_0,\dots,u_{N_t}$ 
(resp. $y_0,\dots,y_{N_t}$), one observes that the objective function can be 
written as $\frac{1}{2}\|R({\bf y},{\bf u})\|^2$ with
\begin{equation}
\label{eq:discreteR}
	R({\bf y},{\bf u}) = 
	\begin{bmatrix}
		\sqrt{\delta_t} M^{1/2} (y_0 - z) \\
		\vdots \\
		\sqrt{\delta_t} M^{1/2} (y_{N_t} - z) \\
		\sqrt{\omega \delta_t} M^{1/2} u_0 \\
		\vdots \\
		\sqrt{\omega \delta_t} M^{1/2} u_{N_t}
	\end{bmatrix} 
	\in \R^{2(N_t +1)N_x}.
\end{equation}

In our experimental setup, we \rev{follow the setup in Troeltzsch and 
Volkwein~\cite{Fredi_Volkwein_2001} and} use $L=T=1$, $N_x=N_t=50$, $\omega=0.05$, and 
$f=0$. We set $z=y_0$ with the first $N_x/2$ coefficients equal to 
$1$ and the others equal to $0$, while the initial control $u_0$ is set to 
the zero vector. \rev{The resulting problem has large residuals.}

\begin{table}[ht!]
	\rev{
	\begin{center}
	\begin{tabular}{llllllll}\toprule
		$\epsR$ &$\epsg$ &\multicolumn{6}{c}{$\theta$} \\ \cmidrule{3-8}
		& &$0$ &$10^{-6}$ &$10^{-4}$ &$10^{-2}$ &$10^{-1}$ &$5.10^{-1}$ \\ \hline
		$10^{-2}$ &$10^{-3}$ &17 &17 &17 &17 &17 &16\\
		$10^{-2}$ &$10^{-4}$ &19 &19 &19 &19 &27 &23\\
		$10^{-2}$ &$10^{-6}$ &21 &21 &23 &39 &49 &43\\
		\bottomrule
	\end{tabular}
	\end{center}
	\vspace*{0.5ex}
	\caption{\rev{Number of PDE solves for six variants of Algorithm~\ref{alg:gn} for the 
	optimal control problem~\eqref{eq:controlpb} using $\nu=0.1$.}}
	\label{tab:burgersnu1em1pde}
	}
\end{table}

\begin{table}[ht!]
	\rev{
	\begin{center}
	\begin{tabular}{llllllllll}\toprule
		$\epsR$ &$\epsg$ &\multicolumn{6}{c}{$\theta$} \\ \cmidrule{3-8}
		& &$0$ &$10^{-6}$ &$10^{-4}$ &$10^{-2}$ &$10^{-1}$ &$5.10^{-1}$ \\ \hline
		$10^{-2}$ &$10^{-3}$ &232 &102 &76  &39 &28 &15\\			
		$10^{-2}$ &$10^{-4}$ &312 &128 &94  &54 &32 &18\\
		$10^{-2}$ &$10^{-6}$ &394 &152 &112 &62 &49 &28\\
		\bottomrule
	\end{tabular}
	\end{center}
	\vspace*{0.5ex}
	\caption{\rev{Jacobian-vector products for six variants of Algorithm~\ref{alg:gn} for the 
	optimal control problem~\eqref{eq:controlpb} using $\nu=0.1$.}}
	\label{tab:burgersnu1em1jvp}
	}
\end{table}

\rev{We first report results for the parameter choice $\nu=0.1$ in 
Tables~\ref{tab:burgersnu1em1pde}--\ref{tab:burgersnu1em1jvp}. For this problem, the 
number of PDE solves worsens significantly as $\theta$ increases, and more iterations 
are needed to correct inexactness in the steps. Meanwhile, the number of Jacobian-vector 
products seemingly decreases with $\theta$. This phenomenon is explained by the large 
number of iterations during which no iteration of conjugate gradient is performed, and 
the method essentially takes a gradient step, hence the slowdown of convergence observed 
on Table~\ref{tab:burgersnu1em1pde}.}

\begin{table}[ht!]
	\rev{
	\begin{center}
	\begin{tabular}{llllllll}\toprule
		$\epsR$ &$\epsg$ &\multicolumn{6}{c}{$\theta$} \\ \cmidrule{3-8}
		& &$0$ &$10^{-6}$ &$10^{-4}$ &$10^{-2}$ &$10^{-1}$ &$5.10^{-1}$ \\ \hline
		$10^{-2}$ &$10^{-3}$ &18 &18 &18  &18  &18  &26\\		
		$10^{-2}$ &$10^{-4}$ &21 &21 &21  &21  &38  &36\\
		$10^{-2}$ &$10^{-6}$ &27 &27 &50  &78  &96  &92\\
		\bottomrule
	\end{tabular}
	\end{center}
	\vspace*{0.5ex}
	\caption{\rev{Number of PDE solves for six variants of Algorithm~\ref{alg:gn} for the 
	optimal control problem~\eqref{eq:controlpb} using $\nu=0.01$.}}
	\label{tab:burgersnu1em2pde}
	}
\end{table}

\begin{table}[ht!]
	\rev{
	\begin{center}
	\begin{tabular}{llllllll}\toprule
		$\epsR$ &$\epsg$ &\multicolumn{6}{c}{$\theta$} \\ \cmidrule{3-8}
		& &$0$ &$10^{-6}$ &$10^{-4}$ &$10^{-2}$ &$10^{-1}$ &$5.10^{-1}$ \\ \hline
		$10^{-2}$ &$10^{-3}$ &331 &159 &111 &69  &39  &19\\		
		$10^{-2}$ &$10^{-4}$ &454 &222 &148 &82  &47  &24\\
		$10^{-2}$ &$10^{-6}$ &700 &304 &187 &107 &76  &52\\
		\bottomrule
	\end{tabular}
	\end{center}
	\vspace*{0.5ex}
	\caption{\rev{Jacobian-vector products for six variants of Algorithm~\ref{alg:gn} for the 
	optimal control problem~\eqref{eq:controlpb} using $\nu=0.01$.}}
	\label{tab:burgersnu1em2jvp}
	}
\end{table}

\rev{Similar observations can be made while choosing a smaller viscosity 
parameter, as illustrated in
Tables~\ref{tab:burgersnu1em2pde}--\ref{tab:burgersnu1em2jvp}. Calculation 
become more tedious as $\nu$ gets smaller, since the instability grows exponentially 
with the evolution time~\cite{Nguyen_2009}. We indeed observe an increase 
in both the number of PDE solves and the number of Jacobian-vector products 
required to reach the desired accuracies. Nevertheless, the results in the 
case $\nu=0.01$ follow a similar trend than those for $\nu=0.1$, thus 
agreeing with the theoretical bounds obtained in Corollary~\ref{co:wccinexPDE} 
(for the number of PDE solves) and Corollary~\ref{co:wccinexhvec} (for the 
number of Jacobian-vector products).}

\section{Conclusion}
\label{sec:conc}

In this paper, we proposed a regularization method for least-squares problems 
subject to implicit constraints, for which we derived complexity guarantees 
that improve over recent bounds derived in the absence of constraints. To this 
end, we leveraged a recently proposed convergence criterion that is 
particularly useful when the optimal solution corresponds to a nonzero objective 
value. Numerical testing conducted on PDE-constrained optimization problems 
showed that the criterion used to derive our complexity bounds bears a 
practical significance.

Our results can be extended in a number of \rev{research directions. Although 
we focused on deriving complexity bounds with objective and derivative 
evaluations as cost units, we believe that those bounds can be combined with 
complexity results on algebraic operations, in order to fully reflect the 
numerical cost of our framework. In addition, our study assumes access to
deterministic evaluations}. We plan on extending our framework to account for
uncertainty in either the objective or the constraints, so as to handle a 
broader range of problems.

\paragraph{Data availability statement} The code necessary to replicate the 
experiments performed in this paper, including the synthetic generation of 
the test problems, is available from the corresponding 
author upon reasonable request.

\rev{
\section*{\rev{Declarations}}
\paragraph{Funding} Funding for C. W. Royer's research 
was partially provided by \emph{Programme Gaspard Monge pour l'Optimisation} 
under the grant OCEAN, by \emph{Agence Nationale de la Recherche} through 
program ANR-19-P3IA-0001 (PRAIRIE 3IA Institute) and by a ``France 2030'' 
support managed by \emph{Agence Nationale de la Recherche} through program 
ANR-23-PEIA-0004 (PEPR PDE-AI). 
\paragraph{Employment} The authors have no affiliations with or involvement 
in any organization or entity with any financial interest or non-financial 
interest in the subject matter or materials discussed in this manuscript.
\paragraph{Financial interests} The authors have no relevant financial or 
non-financial interests to disclose.
}

\bibliographystyle{plain} 
\bibliography{refsLSIC}

\begin{thebibliography}{10}

\bibitem{HSL_2023}
{\em {HSL. A collection of Fortran codes for large scale scientific
  computation}}, November 2023.
\newblock http://www.hsl.rl.ac.uk/.

\bibitem{AAgarwal_NBoumal_BBullins_CCartis_2021}
A.~Agarwal, N.~Boumal, B.~Bullins, and C.~Cartis.
\newblock Adaptive regularization with cubics on manifolds.
\newblock {\em Math. Program.}, 188:85--134, 2021.

\bibitem{HAntil_SDolgov_AOnwunta_2023}
H.~Antil, S.~Dolgov, and A.~Onwunta.
\newblock {TTRISK}: {T}ensor train decomposition algorithm for risk averse
  optimization.
\newblock {\em Numer. Linear Algebra Appl.}, 30:e2481, 2023.

\bibitem{HAntil_DPKouri_MDLacasse_DRidzal_2016EDS}
H.~Antil, D.~P. {Khouri}, M.-D. {Lacasse}, and D.~Ridzal, editors.
\newblock {\em Frontiers in PDE-Constrained Optimization}, volume 163 of {\em
  The IMA Volumes in Mathematics and its Applications}.
\newblock Springer, New York, NY, USA, 2016.

\bibitem{CGBaker_PAAbsil_KAGallivan_2008}
C.~G. {Baker}, P.-A. {Absil}, and K.~A. {Gallivan}.
\newblock An implicit trust-region method on {R}iemannian manifolds.
\newblock {\em IMA J. Numer. Anal.}, 28:665--689, 2008.

\bibitem{MMBaumann_2013}
M.~M. Baumann.
\newblock {Nonlinear Model Order Reduction using POD/DEIM for Optimal Control
  of Burgers’ Equation}.
\newblock Master's thesis, Faculty of Electrical Engineering, Mathematics and
  Computer Science Delft Institute of Applied Mathematics, Delft University of
  Technology, 2013.

\bibitem{EBergou_YDiouane_VKungurtsev_2020}
E.~Bergou, Y.~Diouane, and V.~Kungurtsev.
\newblock Convergence and complexity analysis of a {L}evenberg-{M}arquardt
  algorithm for inverse problems.
\newblock {\em J. Optim. Theory Appl.}, 185:927--944, 2020.

\bibitem{EBergou_YDiouane_VKungurtsev_CWRoyer_2021}
E.~Bergou, Y.~Diouane, V.~Kungurtsev, and C.~W. Royer.
\newblock A nonmonotone matrix-free algorithm for nonlinear
  equality-constrained least-squares problems.
\newblock {\em SIAM J. Sci. Comput.}, 43:S743--S766, 2021.

\bibitem{EBergou_YDiouane_VKungurtsev_CWRoyer_2022}
E.~Bergou, Y.~Diouane, V.~Kungurtsev, and C.~W. Royer.
\newblock A stochastic {L}evenberg-{M}arquardt method for using random models
  with complexity results and application to data assimilation.
\newblock {\em SIAM/ASA J. Uncertain. Quantif.}, 10:507--536, 2022.

\bibitem{NBoumal_2023}
N.~Boumal.
\newblock {\em An introduction to optimization on smooth manifolds}.
\newblock Cambridge University Press, Cambridge, United Kingdom, 2023.

\bibitem{CCartis_NIMGould_PhLToint_2013d}
C.~Cartis, N.~I.~M. {Gould}, and Ph.~L. {Toint}.
\newblock On the evaluation complexity of cubic regularization methods for
  potentially rank-deficient nonlinear least-squares problems and its relevance
  to constrained nonlinear optimization.
\newblock {\em SIAM J. Optim.}, 23(3):1553--1574, 2013.

\bibitem{CCartis_NIMGould_PhLToint_2022}
C.~Cartis, N.~I.~M. {Gould}, and Ph.~L. Toint.
\newblock {\em Evaluation Complexity of Algorithms for Nonconvex Optimization:
  {T}heory, Computation and Perspectives}, volume MO30 of {\em MOS-SIAM Series
  on Optimization}.
\newblock SIAM, 2022.

\bibitem{FECurtis_DPRobinson_CWRoyer_SJWright_2021}
F.~E. {Curtis}, D.~P. {Robinson}, C.~W. {Royer}, and S.~J. {Wright}.
\newblock Trust-region {Newton-CG} with strong second-order complexity
  guarantees for nonconvex optimization.
\newblock {\em SIAM J. Optim.}, 31:518--544, 2021.

\bibitem{Reyes_Kunish_2004}
J.~C. {de los Reyes} and K.~Kunisch.
\newblock A comparison of algorithms for control constrained optimal control of
  the burgers equation.
\newblock {\em CALCOLO}, 41:203 -- 225, 2001.

\bibitem{ESW14}
H.~Elman, D.~Silvester, and A.~Wathen.
\newblock {\em Finite {E}lements and {F}ast {I}terative {S}olvers}, volume
  Second Edition.
\newblock Oxford University Press, 2014.

\bibitem{NIMGould_TRees_JAScott_2017}
N.~I.~M. {Gould}, T.~Rees, and J.~A. {Scott}.
\newblock A higher order method for solving nonlinear least-squares problems.
\newblock Technical Report RAL-TR-2017-010, STFC Rutherford Appleton
  Laboratory, 2017.

\bibitem{NIMGould_TRees_JAScott_2019}
N.~I.~M. {Gould}, T.~Rees, and J.~A. {Scott}.
\newblock Convergence and evaluation-complexity analysis of a regularized
  tensor-{N}ewton method for solving nonlinear least-squares problems.
\newblock {\em Comput. Optim. Appl.}, 73:1--35, 2019.

\bibitem{MHeinkenschloss_2018}
M.~Heinkenschloss.
\newblock {Lecture notes CAAM 454 / 554 – Numerical Analysis II}.
\newblock Rice University, Spring 2018.

\bibitem{MHeinkenschloss_DRidzal_2014}
M.~Heinkenschloss and D.~Ridzal.
\newblock A matrix-free trust-region {SQP} method for equality constrained
  optimization.
\newblock {\em SIAM J. Optim.}, 24:1507--1541, 2014.

\bibitem{MHinze_RPinnau_MUlbrich_SUlbrich_2009}
M.~Hinze, R.~Pinnau, M.~Ulbrich, and S.~Ulbrich.
\newblock {\em Optimization with {PDE} Constraints}.
\newblock Springer Dordrecht, 2009.

\bibitem{Nguyen_2009}
N.~C. {Nguyen}, G.~Rozza, and A.~T. {Patera}.
\newblock Reduced basis approximation and a posteriori error estimation for the
  time-dependent viscous {Burgers'} equation.
\newblock {\em Calcolo}, 46:157--185, 2009.

\bibitem{JNocedal_SJWright_2006}
J.~Nocedal and S.~J. {Wright}.
\newblock {\em Numerical {O}ptimization}.
\newblock Springer Series in Operations Research and Financial Engineering.
  Springer-Verlag, New York, second edition, 2006.

\bibitem{CCPaige_MASaunders_1982}
C.~C. {Paige} and M.~A. {Saunders}.
\newblock {LSQR}: {A}n algorithm for sparse linear equations and sparse least
  squares.
\newblock {\em ACM Trans. Math. Software}, 8:43--71, 1982.

\bibitem{CWRoyer_SJWright_2018}
C.~W. {Royer} and S.~J. {Wright}.
\newblock Complexity analysis of second-order line-search algorithms for smooth
  nonconvex optimization.
\newblock {\em SIAM J. Optim.}, 28:1448--1477, 2018.

\bibitem{LRuthotto_EHaber_2020}
L.~Ruthotto and E.~Haber.
\newblock Deep neural networks motivated by partial differential equations.
\newblock {\em J. Math. Imaging Vision}, 62:352--364, 2020.

\bibitem{Fredi_2010}
F.~{Tr\"{o}ltzsch}.
\newblock {\em Optimal {Control} of {Partial} {Differential} {Equations}:
  {Theory}, {Methods} and {Applications}}.
\newblock American Mathematical Society, 2010.

\bibitem{Fredi_Volkwein_2001}
F.~{Tr\"{o}ltzsch} and S.~{Volkwein}.
\newblock The {SQP} method for the control constrained optimal control of the
  {Burgers} equation.
\newblock {\em ESAIM: Control, Optimisation and Calculus of Variations}, 6:649
  -- 674, 2001.

\bibitem{MUlbrich_2011}
M.~Ulbrich.
\newblock {\em Semismooth {Newton} Methods for Variational Inequalities and
  Constrained Optimization Problems in Function Spaces}, volume MO11 of {\em
  MOS-SIAM Series on Optimization}.
\newblock SIAM, 2011.

\end{thebibliography}

\end{document}